\documentclass[review]{elsarticle}
\usepackage[top=2.5cm, bottom=2.5cm, left=2.5cm, right=2cm]{geometry}
\usepackage{amssymb,latexsym,amscd,amsmath,amsfonts,enumerate,supertabular,amsthm}
\usepackage{graphicx}
\usepackage{color}
\usepackage{tabularx}
\usepackage{subfigure}
\usepackage{enumerate}
\usepackage{epstopdf}
\usepackage{natbib,float}
\usepackage{lscape}
\usepackage{lineno,hyperref}
\newtheorem{theorem}{Theorem}
\newtheorem{lemma}{Lemma}

\newtheorem{remark}{Remark}

\newtheorem{definition}{Definition}
\journal{Elsevier}

\begin{document}

\begin{frontmatter}

\title{A new and simple condition for the global asymptotic stability\\ of a malware spread model on WSNs}


\author[Manh Tuan Hoang]{Manh Tuan Hoang\corref{mycorrespondingauthor}}
\cortext[mycorrespondingauthor]{Corresponding author}
\ead{tuanhm14@fe.edu.vn; hmtuan01121990@gmail.com}
\address[mymainaddress]{Department of Mathematics, FPT University, Hoa Lac Hi-Tech Park,\\ Km29 Thang Long Blvd, Hanoi, Viet Nam}
\begin{abstract}
In a very recent work [J. D. Hern\'andez Guill\'en, A. Mart\'in del Rey, A mathematical model for malware spread on WSNs with population dynamics, Physica A: Statistical Mechanics and its Applications 545(2020) 123609], a novel theoretical model for the spread of malicious code on wireless sensor networks was introduced and analyzed. However, the global asymptotic stability (GAS) of the disease-endemic equilibrium (DEE) point was only resolved partially under technical hypotheses that are not only difficult to be verified but also restrict the space of feasible parameters for the model. In the present work, we use a simple approach to establish the complete GAS of the DEE point without the technical hypotheses proposed in the benchmark work. This approach is  based on a suitable family of Lyapunov functions in combination with characteristics of Volterra-Lyapunov stable matrices. Consequently, we obtain a simple and easily-verified condition for the DEE point to be globally asymptotically stable.  This result provides an important improvement for the results constructed in the benchmark work. In addition, the theoretical findings are supported by numerical and illustrative examples, which show that the numerical results are consistent with the theoretical ones.
\end{abstract}
\begin{keyword}
Malware \sep Wireless sensor networks \sep  Global asymptotic stability \sep Lyapunov stability theory \sep Volterra-Lyapunov stable matrix \sep Dynamical systems
\MSC[2010] 		34C60, 	34D05, 337C75, 7N99 
\end{keyword}
\end{frontmatter}
\section{Introduction}
Nowadays, with the continuous development of the Internet, computer viruses and malware have become a big threat to Internet security and privacy as well as to work and daily life. This leads to urgent requests for effective measures to prevent and control computer viruses and malware. For this purpose, many mathematicians and engineers have built mathematical models, which are based on epidemiological models and the high similarity of the spread of computer viruses and biological ones, to discover characteristics and transmission mechanisms of computer viruses and malware (see, for example, \cite{Akgul, Dubey, HGuillen, Hosseini, Hu, Huang, Ren} and references therein). Consequently, effective measures to prevent and control computer viruses are suggested.\par
In a very recent work \cite{HGuillen}, Hern\'andez Guill\'en and Mart\'in del Rey introduced and analyzed 
a new theoretical model for the spread of malicious code on wireless sensor networks (WSNs). More clearly, a novel mathematical model considering population dynamics, vaccination and reinfection processes was constructed and studied.  In this model, the total population of nodes is classified into the four compartments, that are susceptible devices ($S$), infectious devices ($I$), carrier devices ($C$), and recovered devices ($R$), where
\begin{itemize}
\item susceptible nodes are those that are free of malware;
\item infectious nodes are those that have been reached by the malicious code and it can perform its malicious action;
\item in the case when malware is no able to cause some damage to the host node but the node serves as a transmission vector, the state of the node is that of carrier;
\item  recovered are those (infectious or carrier) devices in which the malicious code has been removed by means of security countermeasures.
\end{itemize}
Under some appropriate assumptions, the following dynamical differential model was formulated
\begin{equation}\label{eq:1a}
\begin{split}
&S'(t) = A + \epsilon R(t) - aI(t)S(t) - v S(t) - \mu S(t),\\
&C'(t) = a(1 - \delta)I(t)S(t) - b_CC(t) - \mu C(t),\\ 
&I'(t) = a\delta I(t)S(t) - b_I I(t) - \mu I(t),\\
&R'(t) = b_C C(t) + b_I I(t) + v S(t) - \epsilon R(t) - \mu R(t),\\
&N'(t) = A - \mu N(t),  
\end{split}
\end{equation}
where $S(t), C(t), A(t), R(t)$ and $N(t)$ stand for the number of $S$, $C$, $A$ and $R$ nodes and the total population at step of time $t$, respectively.\par
The last equation of the model \eqref{eq:1a} implies that the dynamics of \eqref{eq:1a} is qualitatively equivalent to the dynamics of the following limit system (see \cite{HGuillen})
\begin{equation}\label{eq:1b}
\begin{split}
&S'(t) =  A + \epsilon \bigg(\dfrac{A}{\mu} - S(t) - C(t) - I(t)\bigg) - aI(t)S(t) - v S(t) - \mu S(t),\\
&C'(t) = a(1 - \delta)I(t)S(t) - b_CC(t) - \mu C(t),\\ 
&I'(t) = a\delta I(t)S(t) - b_I I(t) - \mu I(t).
\end{split}
\end{equation}
In \cite{HGuillen}, a  feasible set  and the basic reproduction number for the model \eqref{eq:1a} were computed as
\begin{equation}\label{eq:1bc}
\Omega = \big\{(S, C, I) \in \mathbb{R}_+^3|\,\, 0 \leq S + C + I \leq A/\mu\big\}
\end{equation}
and
\begin{equation}\label{eq:1c}
\mathcal{R}_0 = \dfrac{a\delta(A + N_* \epsilon)}{(b_I + \mu)(v + \epsilon + \mu)}, \quad N_* := \dfrac{A}{\mu},
\end{equation}
respectively. It was proved that:
\begin{itemize}
\item The model \eqref{eq:1b} always possesses a disease-free equilibrium (DFE) point $E_0 = (S_0, C_0, I_0)$ for all values of the parameters, where
\begin{equation}\label{eq:1c}
S_0 = \dfrac{A + N_* \epsilon}{v + \epsilon + \mu}, \quad C_0 = 0, \quad I_0 = 0.
\end{equation}
\item A disease-endemic equilibrium point (DEE) point $E^* = (S^*, C^*, I^*)$ exists if and only if $\mathcal{R}_0 > 1$, where $E^*$ is given by
\begin{equation}\label{eq:1d}
\begin{split}
&S^* = \dfrac{b_I + \mu}{a\delta},\\
&C^* = \dfrac{(-1 + \delta)(b_I + \mu)\big[-a\delta(A + N^*\epsilon) + b_I(v + \epsilon + \mu) + \mu(v + \epsilon + \mu)\big]}{a\delta\big[\mu(\epsilon + \mu) + b_I(\epsilon - \delta\epsilon + \mu) + b_C(b_I + \delta\epsilon + \mu)\big]},\\
&I^* = -\dfrac{(b_C + \mu)\big[-a\delta(A + N^*\epsilon) + b_I(v + \epsilon + \mu) + \mu(v + \epsilon + \mu)\big]}{a\big[\mu(\epsilon + \mu) + b_I(\epsilon - \delta\epsilon + \mu) + b_C(b_I + \delta\epsilon + \mu)\big]}.
\end{split}
\end{equation}
\end{itemize}
Importantly, the global asymptotic stability (GAS) of the DFE was fully determined based on the Lyapunov stability theory (see \cite[Theorem 3]{HGuillen}), whereas, the GAS of the DEE point was partially established by using the geometric approach (see \cite[Theorem 4]{HGuillen}). It was shown that
\begin{itemize}
\item The disease-free steady state, $E_0$, is globally asymptotically stable with respect to $\Omega$ if $\mathcal{R}_0 \leq 1$.
\item The endemic steady state, $E^*$, is globally asymptotically stable if $\mathcal{R}_0 > 1$ and the following inequalities hold:
\begin{equation}\label{eq:2a}
b_I + aN_* - v - ac - b_C - \mu - \delta ac < 0,
\end{equation}
\begin{equation}\label{eq:2b}
-\mu - b_C + \epsilon + aN_*\delta < 0.
\end{equation}
\end{itemize}
It is well-known that studying the GAS of dynamical systems governed by differential equations has played an important role in both theory and applications \cite{Allen, Brauer, Khalil, LaSalle, Lyapunov}. On the other hand, the model \eqref{eq:1a} provides many useful applications in information technology and computer science \cite{HGuillen}, this was also mentioned in recent works \cite{Liu1, Liu2}. Therefore, the study of GAS of the model \eqref{eq:1a} is an essential problem with many practical significance. Although the condition \eqref{eq:2b} is easily-verified, it is not a trivial task to verify the condition \eqref{eq:2a} since the parameter $c$ depends implicitly on the parameters of the model \eqref{eq:1a} (see the proof of Theorem 4 in \cite{HGuillen}). Moreover, the conditions \eqref{eq:2a} and \eqref{eq:2b}  restrict the set of feasible parameters for the model \eqref{eq:1a}.\par
Motivated and inspired by the above reasons, in this work we introduce a simple approach to establish the GAS of the DEE point $E^*$ without the conditions \eqref{eq:2a} and \eqref{eq:2b}. The approach is based on a suitable family of Lyapunov functions in combination with properties of Volterra-Lyapunov stable matrices. Although the incorporation of the theory of Volterra–Lyapunov stable matrices into the classical method of Lyapunov functions was first introduced by Liao and Wang to study the GAS of some three- and four-dimensional epidemiological models (see \cite{Liao}), the results in \cite{Liao} required that the resulting matrices $W$ associated with the Volterra–Lyapunov stable matrices $A$ must be constant. However, it is not easy to verify this hypothesis in general. In the present work, by proposing nonstandard techniques and a suitable family of Lyapunov functions and utilizing characteristics of Volterra-Lyapunov stable matrices, we obtain a new and simple condition for the DEE point $E^*$ to be globally asymptotically stable. This condition is not only simpler, but also easier to be verified than the conditions \eqref{eq:2a} and \eqref{eq:2b}. The obtained results provide an important improvement for the results constructed in \cite{HGuillen}. Furthermore, differently from the previous work \cite{Liao}, we use a necessary and sufficient condition for matrices of order 3 to be Volterra-Lyapunov stable instead of a sufficient condition as in \cite{Liao}. As a result, complicated algebraic manipulations are avoided. It should be emphasized that the present approach can be applied to investigate asymptotic stability properties of the model \eqref{eq:1a} in the context of fractional derivatives.\par
The plan of this work is as follows:\\
In Section \ref{sec2}, we provide some concepts and preliminaries. The GAS of the DEE point $E^*$ is studied in Section \ref{sec3}. Section \ref{sec4} reports some numerical simulations. The last section presents some conclusions and open problems.
\section{Preliminaries}\label{sec2}
In this section, we recall from \cite{Cross} some well-known results related to Volterra-Lyapunov stable matrices.\par
Let $A = (a_{ij})$ be a real matrix of order $n$ and $D = diag(d_1, d_2, \ldots, d_n)$ be a real diagonal matrix. We write $A > 0$ when $A$ is a symmetric positive definite and $A < 0$ when $A$ is a symmetric negative definite.
\begin{definition}(\cite[Definition 1]{Cross})
A matrix $A$ is said to be Volterra-Lyapunov stable if there exists a $D > 0$ for which $AD + 
DA^T < 0$. 
\end{definition}
For any subset $1 \leq i_1 < i_2 < \ldots < i_j \leq n$ of the integers $1, 2, \ldots, n$ the
principal submatrix $A_{i_1\ldots i_j}$ of $A$ is obtained by omitting all rows and columns
except those with indices $i_1, i_2, \ldots, i_j$. The corresponding principal minor is
$M_{i_1\ldots i_j} = \det A_{i_1\ldots i_j}$. The minors $M_i$ are written simply $a_{ii}$.
\begin{definition}(\cite[Definition 2]{Cross})
The signed principal minors of $A$ are the quantities $(-1)^{j}M_{i_1\ldots i_j}$.
\end{definition}
\begin{definition}(\cite[Definition 3]{Cross})
Let $P$ denote the class of matrices whose signed principal
minors are all positive and $P_0^+$ the class whose signed principal minors are all
non-negative, with at least one of each order positive.
\end{definition}
\begin{theorem}\label{theoremVL}(\cite[Theorem 4]{Cross})
A $3 \times 3$ real matrix $A$ is Voltera-Lyapunov stable if and only if $A \in P$  and the inequalities
\begin{equation}\label{eq:3a}
p_1(y) = (a_{13}y + a_{31})^2 - 4a_{11}a_{33}y < 0,
\end{equation}
\begin{equation}\label{eq:3b}
p_2(y) = (b_1y + b_2)^2 - 4M_{12}M_{23}y < 0,
\end{equation}
where $b_1 = a_{12}a_{23} - a_{22}a_{13}$ and $b_2 = a_{21}a_{32} - a_{22}a_{31}$, are satisfied simultaneously.
\end{theorem}
\section{Main results}\label{sec3}
In this section, the GAS of the DEE point $E^*$ is studied based on suitable Lyapunov functions and properties of Volterra-Lyapunov stable matrices. Throughout this section, we always assume that the basic reproduction number $\mathcal{R}_0$ is greater than $1$.\par
Since the total population of \eqref{eq:1a} satisfies $N' = A - \mu N$, it is sufficient to consider the following sub-model
\begin{equation}\label{eq:1}
\begin{split}
S' &= A + \epsilon R - aIS - v S - \mu S,\\
I' &= a\delta IS - b_I I - \mu I,\\
R' &= b_C \bigg(\dfrac{A}{\mu} - S - I - R\bigg) + b_I I + v S - \epsilon R - \mu R
\end{split}
\end{equation}
on its positively invariant set $\Omega$ given by \eqref{eq:1bc}. Note that the representation $C = A/\mu - S - I - R$ was used in \eqref{eq:1a} to obtain \eqref{eq:1}. \par
For the model \eqref{eq:1}, let us denote
\begin{equation*}
M(t) := \delta S(t) + I(t), \quad t \geq 0.
\end{equation*}
Then, we obtain an equation for $M$
\begin{equation*}
M' = \delta S' + I' = \delta A + \delta \epsilon R - \delta v S - \delta \mu S - b_I I - \mu I = \delta A + \delta \epsilon R - v(M - I) - \mu M - b_I I.
\end{equation*}
This equation leads to a new model that is equivalent to \eqref{eq:1}
\begin{equation}\label{eq:2}
\begin{split}
&M' = \delta A + \delta \epsilon R - (v + \mu)M + (v - b_I)I,\\
&I' = a(M - I)I - b_I I - \mu I,\\
&R' = b_C \bigg(\dfrac{A}{\mu} - \dfrac{M - I}{\delta} - I - R\bigg) + b_I I + v \dfrac{M - I}{\delta}  - \epsilon R - \mu R.
\end{split}
\end{equation}
A feasible set for the model \eqref{eq:2} can be given by
\begin{equation}\label{eq:3}
\Omega^* = \big\{(M, I, R) \in \mathbb{R}_+^3| 0 \leq M \leq A/\mu,\, 0 \leq I + R \leq A/\mu\big\}.
\end{equation}
Form \eqref{eq:1d}, we deduce that if $\mathcal{R}_0 > 1$ the model \eqref{eq:2} has a unique positive equilibrium point $E_2^{*} = (M_2^{*}, \, I_2^{*},\, R_2^{*})$ given by
\begin{equation}\label{eq:4}
\begin{split}
&M_2^{*} = \dfrac{b_I + \mu}{a} -  \dfrac{(b_C + \mu)\big[-a\delta(A + N^*\epsilon) + b_I(v + \epsilon + \mu) + \mu(v + \epsilon + \mu)\big]}{a\big[\mu(\epsilon + \mu) + b_I(\epsilon - \delta\epsilon + \mu) + b_C(b_I + \delta\epsilon + \mu)\big]},\\
&I_2^{*} = -\dfrac{(b_C + \mu)\big[-a\delta(A + N^*\epsilon) + b_I(v + \epsilon + \mu) + \mu(v + \epsilon + \mu)\big]}{a\big[\mu(\epsilon + \mu) + b_I(\epsilon - \delta\epsilon + \mu) + b_C(b_I + \delta\epsilon + \mu)\big]},\\
&R_2^{*} = \dfrac{A}{\mu} - S^* - I^* - C^*.
\end{split}
\end{equation}
Now, the GAS of $E^*$ for the model \eqref{eq:1a} is equivalent to the GAS of $E_2^*$ for the model \eqref{eq:2}.
Because $E_2^*$ is the unique positive equilibrium point of the model \eqref{eq:2}, we have
\begin{equation}\label{eq:5a}
\begin{split}
&\delta A + \delta \epsilon R_2^* - (v + \mu)M_2^* + (v - b_I)I_2^* = 0,\\
& aM_2^* - aI_2^* = b_I + \mu,\\
& b_C\dfrac{A}{\mu} + \dfrac{v - b_C}{\delta}M_2^* + \bigg(\dfrac{b_C}{\delta} - b_C + b_I - \dfrac{v}{\delta}\bigg)I_2^* - (b_C + \epsilon + \mu)R_2^* = 0.
\end{split}
\end{equation} 
By using \eqref{eq:5a}, the model \eqref{eq:2} can be rewritten in the form
\begin{equation}\label{eq:5}
\begin{split}
M' &= \Big[\delta A + \delta \epsilon R - (v + \mu)M + (v - b_I)I\Big] - \Big[\delta A + \delta \epsilon R_2^* - (v + \mu)M_2^* + (v - b_I)I_2^*\Big]\\
&= -(v + \mu)(M - M_2^*) + (v - b_I)(I - I_2^*) + \delta \epsilon (R - R_2^*),\\
I'&= I\big[(aM - aI) - (b_I + \mu)\big] =  I\big[(aM - aI) - (aM_2^* - aI_2^*)\big],\\
&= I\Big[a(M - M_2^*)  - a(I - I_2^*)\Big],\\
R'&= \bigg[ b_C\dfrac{A}{\mu} + \dfrac{v - b_C}{\delta}M + \bigg(\dfrac{b_C}{\delta} - b_C + b_I - \dfrac{v}{\delta}\bigg)I - (b_C + \epsilon + \mu)R \bigg]\\
&- \bigg[ b_C\dfrac{A}{\mu} + \dfrac{v - b_C}{\delta}M_2^* + \bigg(\dfrac{b_C}{\delta} - b_C + b_I - \dfrac{v}{\delta}\bigg)I_2^* - (b_C + \epsilon + \mu)R_2^* \bigg]\\
&= \dfrac{v - b_C}{\delta}(M - M_2^*) + \bigg(\dfrac{b_C}{\delta} - b_C + b_I - \dfrac{v}{\delta}\bigg)(I- I_2^*) - (b_C + \epsilon + \mu)(R - R_2^*).
\end{split}
\end{equation}
Consider a family of Lyapunov functions defined by
\begin{equation}\label{eq:6}
V(M, I, R) = \tau_1\big(R - R_2^*\big)^2 + \tau_2 \big(M - M_2^*\big)^2 + 2\tau_3 \bigg(I - I_2^* - I_2^*\ln\dfrac{I}{I_2^*}\bigg),
\end{equation}
where $\tau_1, \tau_2, \tau_3$ are positive undetermined coefficients. The derivative of the function $V$ along solutions of \eqref{eq:5} satisfies
\begin{equation}\label{eq:7}
\begin{split}
\dfrac{dV}{dt} &= \dfrac{d V}{d R}\dfrac{dR}{dt} + \dfrac{d V}{d M}\dfrac{dM}{dt} + \dfrac{d V}{d I}\dfrac{dI}{dt}\\
&=  2\tau_1 \big(R - R_2^*\big)R' + 2\tau_2 \big(M - M_2^*\big)M' + 2\tau_3 \dfrac{I - I_2^*}{I}I'\\
& =  2\tau_1\dfrac{v - b_C}{\delta}(M - M_2^*)(R - R_2^*) + 2\tau_1\bigg(\dfrac{b_C}{\delta} - b_C + b_I - \dfrac{v}{\delta}\bigg)(R - R_2^*)(I- I_2^*)- 2\tau_1(b_C + \epsilon + \mu)(R - R_2^*)^2\\
&-2\tau_2(v + \mu)(M - M_2^*)^2 + 2\tau_2\delta \epsilon(M - M_2^*)(R - R_2^*) + 2\tau_2(v - b_I)(M - M_2^*)(I - I_2^*)\\
&+ 2\tau_3a(I - I_2^*)(M - M_2^*) - 2\tau_3a(I - I_2^*)^2,
\end{split}
\end{equation}
or equivalently to
\begin{equation}\label{eq:8a}
\dfrac{dV}{dt} = X\big(QD + DQ^T)X^T,
\end{equation}
where $X, Q$ and $D$ are matrices given by
\begin{equation}\label{eq:8}
X = \big[R - R_2^* \quad M - M_2^* \quad I - I_2^*\big],\quad
Q = 
\begin{pmatrix}
-(b_C + \epsilon + \mu)&\delta \epsilon&0\\
&&\\
\dfrac{v - b_C}{\delta}&-(v + \mu)&a\\
&&\\
\dfrac{b_C}{\delta} - b_C + b_I - \dfrac{v}{\delta}&v - b_I&-a
\end{pmatrix},\quad
D = 
\begin{pmatrix}
\tau_3&0&0\\
0&\tau_1&0\\
0&0&\tau_2
\end{pmatrix}.
\end{equation}
Our main objective is to determine conditions guaranteeing that \textit{the matrix $Q$ is Volterra-Lyapunov stable.}
\begin{lemma}\label{mainlemma1}
The matrix $Q$ given by \eqref{eq:8} belongs to the class $P$.
\end{lemma}
\begin{proof}
We need to show that all signed principal minors of $Q$ are positive. Indeed,
\begin{equation}\label{eq:11}
\begin{split}
&M_1= -(b_C + \epsilon + \mu) < 0,\\
&M_2= -(v + \mu) < 0,\\
&M_3= -a < 0,\\
&M_{12}= (b_C + \epsilon + \mu)(v + \mu) - \dfrac{v - b_C}{\delta}\delta \epsilon = (b_C + \mu)(b_C + \mu + \epsilon) > 0,\\
&M_{13}= (b_C + \epsilon + \mu)a > 0,\\
&M_{23}= (v + \mu)a - (v - b_I)a =  a(b_I + \mu) > 0,\\
&M_{123} = \det Q = -a\Big[(b_I + \mu)(b_C + \epsilon + \mu) - \delta\epsilon(b_I - b_C)\Big] < 0.
\end{split}
\end{equation}
The last inequality of \eqref{eq:11} is obtained because $0 < \delta < 1$. Therefore, we have that all the signed principal minors of $Q$ are positive, or equivalently, $Q \in P$. The proof is complete.
\end{proof}
We now show that the inequalities \eqref{eq:3a} and \eqref{eq:3b} can be satisfied simultaneously for the matrix $Q$. For this purpose, we introduce the following hypothesis:
\begin{equation}\label{maincondition}
\begin{split}
&\dfrac{\Big(b_Ca\epsilon - b_Ca\delta\epsilon + b_Ia\delta\epsilon - {va\epsilon}\Big)^2}{4(b_C + \epsilon + \mu)(a + b_I + \mu)}\\
&< (b_C + \mu)(b_C + \mu + \epsilon)(b_I + \mu)a + (v + \mu)a\Big[(b_I + \mu)(b_C + \epsilon + \mu) - \delta\epsilon(b_I - b_C)\Big]\\
&+2\sqrt{(b_C + \mu)(b_C + \mu + \epsilon)(b_I + \mu)a(v + \mu)a\Big[(b_I + \mu)(b_C + \epsilon + \mu) - \delta\epsilon(b_I - b_C)\Big]}.
\end{split}
\end{equation}
%
\begin{lemma}\label{mainlemma2}
The inequalities \eqref{eq:3a} and \eqref{eq:3b} are satisfied simultaneously for the matrix $Q$ if the condition \eqref{maincondition} holds.
\end{lemma}
\begin{proof}
For the matrix $Q$, we have
\begin{equation}\label{eq:13}
p_1(y) = \bigg(\dfrac{b_C}{\delta} - b_C + b_I - \dfrac{v}{\delta}\bigg)^2 - 4(b_C + \epsilon + \mu)(a + b_I + \mu)y.
\end{equation}
So, $p_1(y) < 0$ whenever
\begin{equation}\label{eq:14}
y \in \Omega_1 := (y_{p_1}^*, \infty), \quad y_{p_1}^* := \dfrac{\bigg(\dfrac{b_C}{\delta} - b_C + b_I - \dfrac{v}{\delta}\bigg)^2}{4(b_C + \epsilon + \mu)(a + b_I + \mu)}.
\end{equation}
On the other hand, the polynomial $p_2(y)$ is given by \cite{Cross}
\begin{equation*}
p_2(y) = b_1^2 y^2 - 2\big(M_{12}M_{23} + a_{22}\det Q\big)y + b_2^2,
\end{equation*}
where $b_1 = q_{12}q_{23} - q_{22}q_{13}$ and $b_2 = q_{21}q_{32} - q_{22}q_{31}$. So, its discriminant is
\begin{equation*}
\Delta = 16a_{22}M_{12}M_{23}\det Q.
\end{equation*}
It follows from \eqref{eq:11} that $\Delta > 0$. Consequently, $p_2(y) > 0$ whenever
\begin{equation*}
y \in \Omega_2 = \bigg(\dfrac{M_{12}M_{23} + q_{22}\det Q - 2\sqrt{M_{12}M_{23}q_{22}\det Q}}{b_1^2},\,\,\,\dfrac{M_{12}M_{23} + q_{22}\det Q + 2\sqrt{M_{12}M_{23}q_{22}\det Q}}{b_1^2}\bigg).
\end{equation*}
The inequalities \eqref{eq:3a} and \eqref{eq:3b} are satisfied simultaneously for the matrix $Q$ if and only if $\Omega_1 \cap \Omega_2 \ne \emptyset$. This condition is equivalent to
\begin{equation*}\label{eq:81}
y_{p_1}^* < \dfrac{M_{12}M_{23} + q_{22}\det Q + 2\sqrt{M_{12}M_{23}q_{22}\det Q}}{b_1^2},
\end{equation*}
or equivalently,
\begin{equation}\label{eq:82}
y_{p_1}^* := \dfrac{\bigg(\dfrac{b_C}{\delta} - b_C + b_I - \dfrac{v}{\delta}\bigg)^2}{4(b_C + \epsilon + \mu)(a + b_I + \mu)} < \dfrac{M_{12}M_{23} + q_{22}\det Q + 2\sqrt{M_{12}M_{23}q_{22}\det Q}}{(a\delta\epsilon)^2}.
\end{equation}
The condition \eqref{eq:82} is equivalent to \eqref{maincondition}. This completes the proof.
\end{proof}
Combing Lemmas \ref{mainlemma1} and \ref{mainlemma2} with Theorem \ref{theoremVL} we obtain:
\begin{theorem}\label{maintheorem1}
The matrix $Q$ given by \eqref{eq:8} is Volterra-Lyapunov stable whenever the condition \eqref{maincondition} hold.
\end{theorem}
As an important consequence of Theorem \ref{maintheorem1}. The GAS of $E_2^*$ of the model \eqref{eq:2} is obtained.
\begin{theorem}\label{maintheorem2}
The DEE point $E_2^*$ of the model \eqref{eq:2} is globally asymptotically stable if the condition \eqref{maincondition} is satisfied.
\end{theorem}
\begin{proof}
First, we recall from \eqref{eq:8a} that
\begin{equation*}
\dfrac{dV}{dt} = X(QD + DQ^T)X.
\end{equation*}
Since the matrix $Q$ is Volterra-Lyapunov stable, there exists a diagonal matrix $D > 0$ for which $QD + DQ^T < 0$. Consequently, there always exist positive real numbers $\tau_1, \tau_2$ and $\tau_3$ such that the function $V$ defined by \eqref{eq:6} satisfies the Lyapunov stability theorem \cite{LaSalle, Lyapunov}. Hence, the GAS of $E_2^*$ is proved. The proof is completed.
\end{proof}
Because the GAS of the models \eqref{eq:2} and \eqref{eq:1a} is equivalent, we obtain the GAS of the original model \eqref{eq:1a}.
\begin{theorem}\label{theoremnew}
The endemic steady state $E^*$ of the model \eqref{eq:1a} is globally asymptotically stable provided that $\mathcal{R}_0 > 1$ and the condition \eqref{maincondition} is satisfied.
\end{theorem}
\begin{remark}
\begin{enumerate}[(i)]
\item It is easy to verify the condition \eqref{maincondition} since it is fully explicit. Moreover, \eqref{maincondition} is satisfied automatically when $\dfrac{b_C}{\delta} - b_C + b_I - \dfrac{v}{\delta} = 0$.
\item Since the matrix $Q$ given by \eqref{eq:8} is constant, the resulting matrix $D$ is also constant. Meanwhile,   the results in \cite{Liao} required that the matrices $D$ associated with the Volterra–Lyapunov stable matrices $Q$ must be constant. In general, it is difficult to verify this condition.
\item The representation of the derivative of the function $V$ given by \eqref{eq:8a} is not unique. There are $6$ ways to represent $dV/dt$ in the matrix form depending on the order of $M - M_2^*, I - I_2^*$ and $R - R_2^*$. For example, we can write
\begin{equation*}\label{eq:85}
\dfrac{dV}{dt} = X(DQ + QD^T)X^T,
\end{equation*}
where
\begin{equation*}
X = \big[M - M_2^* \quad I - I_2^*\quad R - R_2^*\big],\quad
Q = 
\begin{pmatrix}
-(v + \mu)&a&\dfrac{v - b_C}{\delta}\\
&&\\
v - b_I&-a&\dfrac{b_C}{\delta} - b_C + b_I - \dfrac{v}{\delta}\\
&&\\
\delta\epsilon&0&-(b_C + \epsilon + \mu)
\end{pmatrix},
\quad
D = 
\begin{pmatrix}
\tau_2&0&0\\
0&\tau_3&0\\
0&0&\tau_1
\end{pmatrix}
\end{equation*}
or
\begin{equation*}
X = \big[M - M_2^* \quad R - R_2^*\quad I - I_2^*\big], \quad
Q = 
\begin{pmatrix}
-(v + \mu)&\dfrac{v - b_C}{\delta}&a\\
&&\\
\delta\epsilon&-(b_C + \epsilon + \mu)&0\\
&&\\
v - b_I&\dfrac{b_C}{\delta} - b_C + b_I - \dfrac{v}{\delta}&-a
\end{pmatrix},
\quad
D = 
\begin{pmatrix}
\tau_2&0&0\\
0&\tau_1&0\\
0&0&\tau_3
\end{pmatrix}.
\end{equation*}
However, these representations are not convenience to determine conditions for the matrix $Q$ to be Lyapunov-Volterra stable.
\end{enumerate}
\end{remark}
\section{Numerical examples}\label{sec4}
In this section, some numerical examples, using the data in Table \ref{table1}, are performed to support the theoretical results.\par
We now solve numerically the differential equation model \eqref{eq:1a} by applying the classical fourth-order Runge-Kutta (RK4) method (see \cite{Ascher})  with $h = 10^{-5}$. The solutions of the model generated by the RK4 method for $t \in [0, \,\, 160]$ are depicted in Figures \ref{Fig:1} and \ref{Fig:2}.  In these figure, each blue curve represents a phase space corresponding to a specific initial data, the yellow arrows show the evolution of the model and the red circle indicates the location of the disease-endemic equilibrium point.\par
It is clear that in both cases the solutions are stable and converge to the equilibrium point $E^*$. In other words, $E^*$ is globally asymptotically stable. Also, the numerical results show that the conditions \eqref{eq:2a} and \eqref{eq:2b} are only technical ones.
\begin{landscape}
\begin{table}[H]
\begin{center}
\caption{The parameters used in numerical examples.}\label{table1}
\begin{tabular}{ccccccccccccccc}
\hline
Case&$A$&$\epsilon$&$a$&$v$&$\mu$&$\delta$&$b_I$&$b_C$&$\mathcal{R}_0$&$E_* = (S^*, I^*, R^*)$&Remark\\
\hline
$1$&2&0.004&0.008&0.05&0.01&0.9&0.1&0.005&2.8636&$(15.2796,\,\, 14.0548,\,\, 158.6942)$&The condition \eqref{eq:2b} is not satisfied\\
&&&&&&&&&&&The condition \eqref{maincondition} is satisfied\\
\hline
2&2&0.002&0.001&0.03&0.01&0.5&0.01&0.05&1.4286&$(39.7946,\,\,  17.1194,\,\, 137.2670)$&The condition \eqref{eq:2b} is not satisfied\\
&&&&&&&&&&&The condition \eqref{maincondition} is not satisfied\\
\hline
\end{tabular}
\end{center}
\end{table}
\end{landscape}
\begin{figure}[H]
\centering
\includegraphics[height=10.1cm,width=15cm]{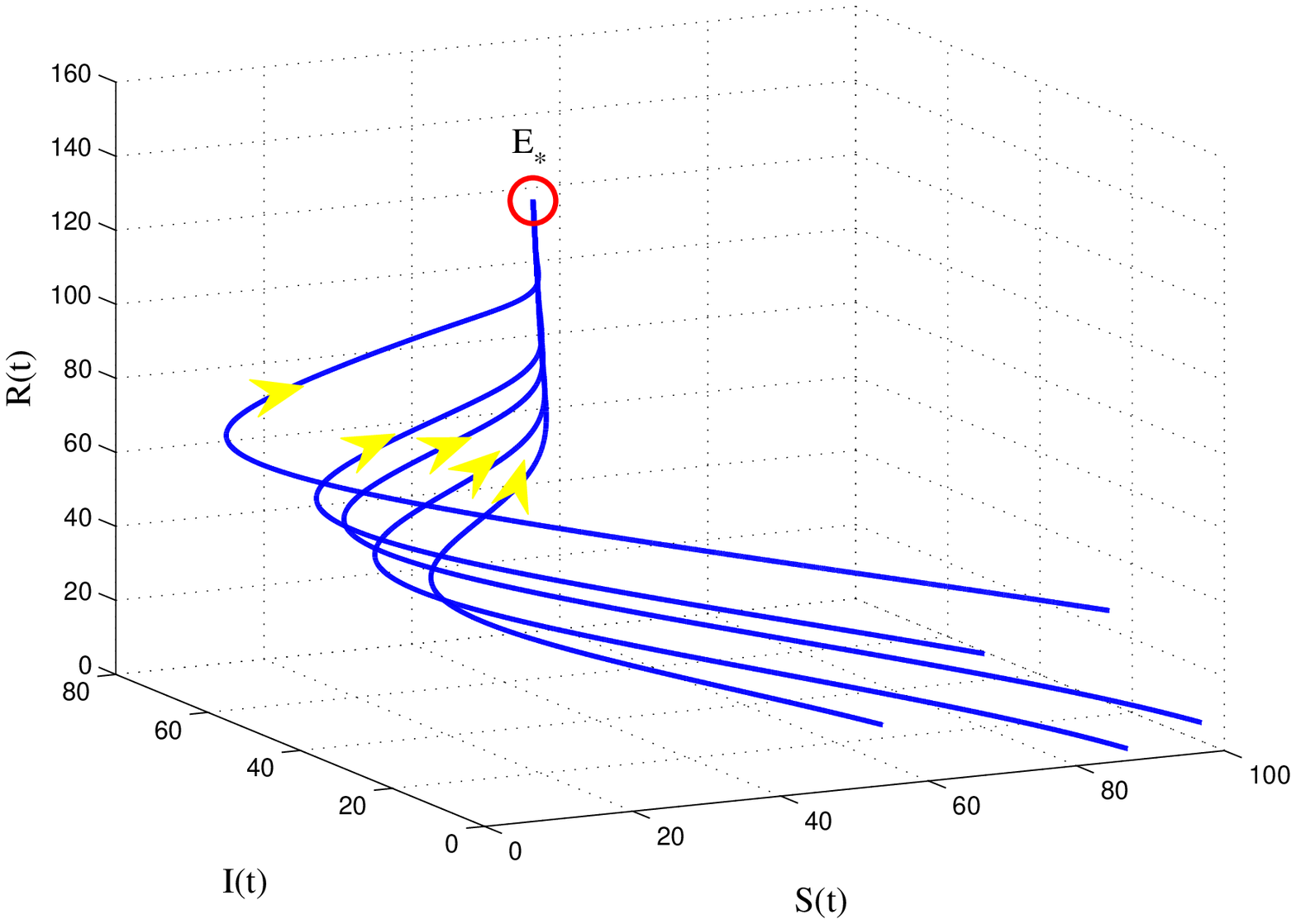}
\caption{The phase spaces for the model \eqref{eq:1a} in Case 1.}\label{Fig:1}
\end{figure}
\begin{figure}[H]
\centering
\includegraphics[height=10.1cm,width=15cm]{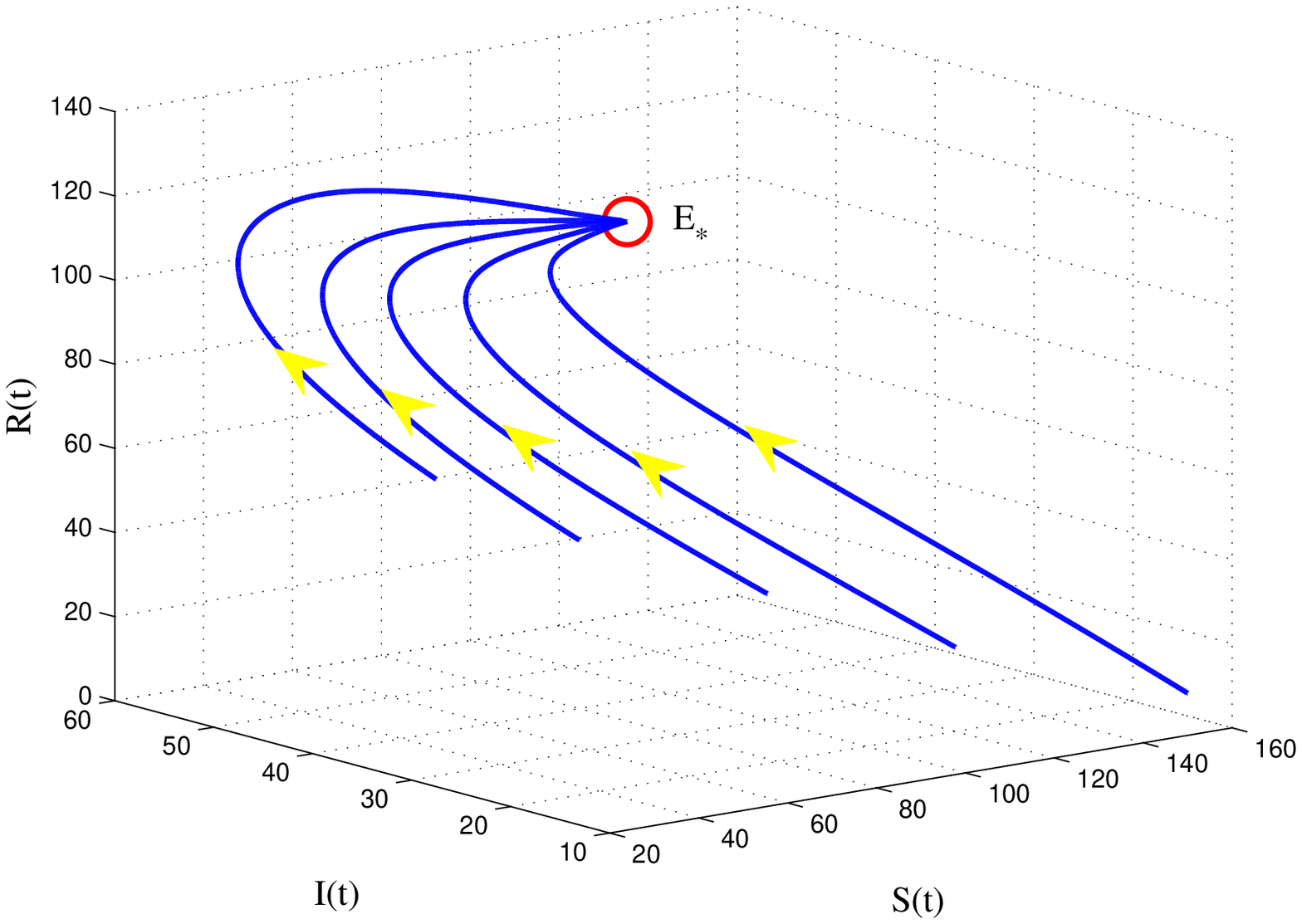}
\caption{The phase spaces for the model \eqref{eq:1a} in Case 2.}\label{Fig:2}
\end{figure}
\section{Conclusions and some open problems}\label{sec5}
In this work, we have used a simple approach to establish  the complete GAS of the SCIRS model \eqref{eq:1a} without the technical hypotheses proposed in \cite{HGuillen}.  This approach is  based on suitable Lyapunov functions in combination with characteristics of Lyapunov-Volterra stable real matrices. The main result is that we obtain a simple and easily-verified condition for the endemic equilibrium point to be globally asymptotically stable.  This result provides an important improvement for the results constructed in \cite{HGuillen}. The obtained theoretical findings have been supported and illustrated by a set of numerical simulations.\par
Recently, some researchers have applied the Mickens' methodology \cite{Mickens1, Mickens2, Mickens3, Mickens4} to formulate dynamically consistent nonstandard finite difference schemes for solving computer viruses and malware propagation models (see, for instance, \cite{Hoang1, Hoang2, Hoang3, MVaquero1, MVaquero2}). This is an important problem with many useful applications in real-world situations. Following these works, we can also apply the Mickens' methodology to propose nonstandard numerical schemes for the model \eqref{eq:1a}. Then, the stability problem for proposed nonstandard numerical schemes will be posted. This is not a simple problem because the equilibria of the continuous model are not only locally asymptotically stable but also globally asymptotically stable. However, the present approach can be considered as a promising solution for the problem.\par
In recent years, many researchers have constructed and developed the Lyapunov stability theory for fractional dynamical systems (see, for instance, \cite{Agarwal, Aguila-Camacho, Duarte-Mermoud, Li1, Li2, VargasDeLeon}). Hence, the present approach can be combined with the Lyapunov stability theory for fractional dynamical system to investigate asymptotic stability properties of the model \eqref{eq:1a} in the context of fractional derivatives such as the Caputo fractional derivative \cite{Caputo}, the Riemann-Liouville fractional derivative \cite{Diethelm, Podlubny}, the Caputo-Fabrizio fractional derivative \cite{Caputo1}, the new fractional derivative involving the normalized sinc function without singular kernel \cite{Yang}, the new fractional derivatives with non-local and non-singular kernel \cite{Atangana} and so on.\par
In the near future, we will study the use of the approach in the present work for resolving the GAS problem for dynamical systems governed by differential and difference equations. Also, the construction of dynamically consistent nonstandard finite difference schemes for the model \eqref{eq:1a} in particular and for computer viruses and malware propagation models in general will be considered.

\end{document}